\documentclass[12pt]{article}
\RequirePackage[OT1]{fontenc}
\RequirePackage{amsthm,amsmath}
\usepackage{soul}
\usepackage[numbers,sort&compress]{natbib}
\RequirePackage[colorlinks=black,citecolor=black,urlcolor=blue]{hyperref}
\usepackage{amsfonts}
\usepackage{amssymb}
\usepackage{graphicx}
\usepackage{setspace}
\usepackage{hyperref}
\newtheorem{theorem}{Theorem}[section]

\newtheorem{comment}{Comment}[section]

\newtheorem{definition}{Definition}[section]

\numberwithin{equation}{section}
\newcommand{\bet}{{\boldsymbol \eta}}

\def\by{{\boldsymbol y}}

\def\b1{{\mathbf 1}}

\def\bz{{\boldsymbol z}}

\def\bof{{\boldsymbol f}}
\def\bx{{\boldsymbol x}}

\def\cM{{\mathcal M}}

\def\cF{{\mathcal F}}

\def\cL{{\mathcal L}}

\def\cI{{\mathcal I}}
\def\cJ{{\mathcal J}}
\def\cG{{\mathcal G}}
\def\bbR{{\mathbb R}}

\def\bbP{{\mathbb P}}

\graphicspath{%
    {converted_graphics/}
    {/}
}

\begin{document}

\title{A predictive approach to generalized\\arithmetic means} 
\author{Henryk Gzyl \\
Centro de Finanzas IESA, Caracas, Venezuela.\\
henryk.gzyl@iesa.edu.ve}

\date{}
 \maketitle

\setlength{\textwidth}{4in}

\vskip 1 truecm
\baselineskip=1.5 \baselineskip \setlength{\textwidth}{6in}
\begin{abstract}
The goal of this note is to provide a geometric setting in which  generalized arithmetic means are best predictors in an appropriate metric. This characterization provides a geometric interpretation to the concept of certainty equivalent. Besides that, in this geometric setting there also exists the notion of conditional expectation as best predictor given prior information. This leads to a notion of conditional preference and to the notion of conditional certainty equivalent, which turns out to be consistent with the notion of fair pricing.

\end{abstract}

 
\begin{spacing}{0.05}
   \tableofcontents
\end{spacing}

\section{Introduction and Preliminaries} 
Generalized arithmetic means are about a century old. The notion seems to have appeared in connection with the concept of certain equivalent. Some review papers, emphasizing the connection of generalized means to utility theory are Muliere and Parmigiano's \cite{MuPar} and Pagani's \cite{Pag}.\footnote{Beware of typos}Consider Matkowski and P\'ales \cite{MP} or P\'ales \cite{Pal}, in which a large number of the original papers are reviewed. Much of the work cited there, for example Bemporad's \cite{Bem}, Bonferroni's \cite{Bon}, de Finetti's \cite{deF}, Kolmogorov's \cite{Ko} and Nagumo \cite{Nag}, to cite some of the earliest, deal with the issue of characterizing such means. Interesting as well is the point of view of Fishburn's \cite{Fi}. To develop the motivational connection, consider the definition of certain equivalent.
\begin{definition}\label{util}
A utility function $u:\cI\to\bbR$ is a strictly increasing, continuous function defined on some interval $\cI$ of the real line. Suppose that $\{x_1,...,x_n\}$ denote the future values of some (random) future payoff, and suppose that they are equi-probable. The certain equivalent of this cash flow is defined by
 \begin{equation}\label{certeq}
c = u^{-1}\Big(\frac{1}{n}\sum_{k=1}^n u(x_k)\Big).
\end{equation}
\end{definition}
The notion certain equivalent of a cash flow is used in finance and in insurance either for pricing purposes or for decision making. It is interpreted as saying that an investor is indifferent between receiving the certain equivalent $c$ or the random cash flow. See Eeckhoudt et al. \cite{EGS} for more on this. The concept of expected utility is part of the more elaborate notion of ambiguity aversion, see \cite{E}, \cite{S} and \cite{GL} for example, but that is a subject that we do not deal with in this work. 

Certainly, when $u(x)=x$ it coincides with the standard arithmetic mean. This suggests the following generalization.
\begin{definition}\label{util}
Let $f:\cI\to\bbR$ be a strictly increasing, continuous function defined on some interval $\cI$ of the real line. The generalized arithmetic $f$-mean of a collection $\{x_1,...,x_n\}\subset \cI$ is defined by
\begin{equation}\label{arim}
m_f(\{x_1,...,x_n\}) = f^{-1}\Big(\frac{1}{n}\sum_{k=1}^n f(x_k)\Big).
\end{equation}
\end{definition}
\begin{comment}\label{com1}
For the definition of generalized arithmetic mean it suffices that $f$ be strictly monotone so that $f{^-1}$ is well defined. We the strictly increasing case relates to utility theory be reinterpreting the function. Besides, when we want to think of $f$ as a utility function, we shall change its name too $u.$
\end{comment}
The references cited a few paragraphs above deal with the properties, characterization and extensions of the notion of utility function. In this work we go in a different direction. mentioned that when $f$ is the identity map, (\ref{arim}) is the standard arithmetic mean. The arithmetic mean is the number closest to  $\{x_1,...,x_n\}$ in the standard Euclidean distance. 

This suggests the following question: Does there exist a distance on $\bbR$ or $\bbR^d$ for that mater, in which (\ref{arim}) is a ``best predictor''? That is, does there exist a distance on $\cI$ for which (\ref{arim}) is the point in $\cI$ closest to $\{x_1,...,x_n\}$?  

The aim of this note is to prove that the answer is in the affirmative, and having done that, we shall then explore the properties of such $f-$mean. When instead of set of numbers we consider $\cI-$valued random variables, we are led to the notion of $f-$expected values and an $f$-conditional expectation. At this point we mention that the best predictor that we obtain below, coincides with the notion of conditional certainty equivalent proposed by Fritelli and Maggis \cite{FM} in a different setting. We shall see as well that the notion of empirical $f$-mean and $f$-Law of Large Numbers obtain as well in this setup. 

The remainder of this note is organized as follows. In Section 2 we show how to define a metric on the convex set $\cM=\cI^d$ ($d\geq1$) in such a way that the appropriate extension of (\ref{arim}) is the ``best predictor'' of a set of points in $\cM.$ In Section 3 we consider $\cM$-valued random variables and examine the notion of $f$-expected mean and $f$-conditional expected mean. If we think of $f$ as a utility function, then the $f$-expected mean coincides with the standard notion of certain equivalent. In Section 4 we consider the notion of conditional preference and explain why it is consistent.In Sections  5 we examine some simple aspects of this issue, in particular we introduce the notion of sequential conditional certainty values and prove that they are martingales, therefore it makes sense to thing of them as prices. It is in Section 6 that we consider the empirical sample generalized $f$-mean, verify that it is an unbiased estimator of the generalized mean, and verify that the Law of Large Numbers (LLN) applies and establishes a convergence of the estimator to the mean. We consider some extensions in Section 6 and sum up the contents in the last section. 

\section{The $f$-distance on $\cM.$}
Let $f:\cI\to\bbR$ be a strictly increasing function defined on an interval $\cI\subset\bbR,$ and for not to use new symbols, let $\cM=\cI^d$ and $\bof:\cM\to\bbR^d$ be $d$-dimensional map defined by $\bof^i(\bx)=f(x^i)$ for $i=1,...,d.$ This map is a continuous homeomorphism between $\cM$ and $\bof(\cM).$ We shall use super-scripts to label components and sub-scripts to label entities of the same kind. Thus $x^i$ is the $i$-th component of the $d$-vector $\bx.$ Define the $d_f$-distance between points $\bx,\by\in\cM$ by
\begin{equation}\label{fdist}
d_f(\bx,\by)^2 = \sum_{k=1}^d\big(f(x^k)-f(y^k)\big)^2 = \|\bof(\bx)-\bof(\by)\|^2.
\end{equation}
The fact that $\bof$ is bijective clearly implies that
$$d_f(\bx,\by) = 0 \Leftrightarrow \bx = \by$$
 Thus, $d_f$ is clearly a distance because it is built upon the Euclidean distance. We might call it the $f$-distorted distance on $\cM.$ 
\begin{theorem}\label{fmean}
Let $\{\bx_1,...,\bx_n\}\subset\cM$ be some set of points. The $\by\in\cM$ satisfying
\begin{equation}\label{min1}
\by = argmin\{\sum_{j=1}^Nd_f(\bx_j,\bz) | \bz\in\cM\}
\end{equation}
is given by
\begin{equation}\label{vmean}
\by = \bof^{-1}\Big(\frac{1}{N}\sum_{j=1}^N\bof(\bx_j)\Big).
\end{equation}
Identity (\ref{vmean}) is to be read componentwise.
\end{theorem}
\begin{proof}
It is rather simple, because the result holds componentwise in $\bbR^d.$ Note that, since $\cI$ is convex, so are $\cM$ and $\bof(\cM).$ Therefore means of points in these sets stay in the sets. The existence of a point $\bet\in\bbR^d$ minimizing
$$\sum_{j=1}^N(\bof(\bx_j)-\bz)^2$$
is clear, as well as that the point is given by
$$\bet = \frac{1}{N}\sum_{j=1}^N\bof(\bx_j) \in \bof(\cM).$$
Now set $\bet=\bof(\by)$ to obtain (\ref{vmean}).
\end{proof}

\section{Generalized $f$-expected value and variance}
Consider to begin with, a probability space $(\Omega,\cF,\bbP),$ and for not to overburden the notation we shall consider $\cI$-valued random variables only, which will be denoted by capital letters: $X:\Omega\to \cI$ for example. We reserve $E$ to denote expected values with respect to $\bbP.$ 
\begin{definition}\label{efmean}
With the notations introduced above, we say that the random variable is $f$-integrable or (p-integrable) whenever $E[|f(X)|^p]<\infty.$ In this case we write $X\in\cL_f^p.$ Or $\cL_f^p(\cF)$ if we need to specify the $\sigma$-algebra. The $f$-distance between two random variables $X,Y \in \cL_f^2$ is defined by
\begin{equation}\label{fdist2}
d_f(X,Y)^2 = E[\big(f(X) - f(Y)\big)^2] = \int_\Omega\big(f(X) - f(Y)\big)^2d\bbP.
\end{equation}
\end{definition}

It is again easy to verify that this is a distance if we identify random variables up to sets of $\bbP-$measure equal to $0.$ The best predictors in this distance come up as follows.
\begin{theorem}\label{bestpred1}
Let $X\in\cL_f^2$ and let $\cG$ be a sub$-\sigma$-algebra of $\cF.$ Then there is a unique (up to $\bbP-$null sets) $\cG$-measurable random variable $X^*$ such that
\begin{equation}\label{bestpred2}
X^* = argmin\{d_f(X,Y)^2 | Y\in \cL_f^2(\cG)\}
\end{equation}
It can be computed as
\begin{equation}\label{bestpred3}
X^* = f^{-1}\Big(E[f(X)|\cG]\Big).
\end{equation}
\end{theorem}
\begin{proof}
The proof follows the standard format. See Jacod and Protter's \cite{JP} for example. Note that $f(X)$ is a $\cF$-measurable, $f(\cI)$-valued, square integrable random variable, therefore, there exists a $\cG$-measurable random variable $Y^*$ unique up to $\bbP$-null sets, which minimizes $E[\big(f(X)-Y)^2$ over $Y\in\cL^2(\cG).$ This random variable is given by $Y^*=E[f(X)|\cG].$ Therefore, defining $X^*=f^{-1}(Y^*)$ we obtain a random variable with the desired properties.
\end{proof}
From now on we denote $X^*$ by:
\begin{equation}\label{bestpred4}
E_f[X|\cG]= f^{-1}\Big(E[f(X)|\cG]\Big).
\end{equation}
When $\cG=\{\emptyset,\Omega\},$ that is when no information is given, the constant that best approximates $X$ is its generalized arithmetic $f$-mean value:  
\begin{equation}\label{bestpred5}
E_f[X] = f^{-1}\Big(E[f(X)]\Big).
\end{equation}
The $f$-conditional expectation shares many of the standard properties of the standard conditional expectation. Except for the linearity and Jensen's inequality, the monotonicity and continuity properties hold. Interesting for us here is that the ``tower'' property holds, that is:
\begin{theorem}
Let $\cG_1\subset\cG_2\subset\cF$ be sub-$\sigma$-algebras, and let $X$ be in $\cL_f^1.$ Then
\begin{equation}\label{tower}
E_f\big[E_f[X|\cG_2]|\cG_1\big] = E_f[X|\cG_1].
\end{equation}
When $\cG_1=\{\emptyset,\Omega\},$ the previous identity yields (\ref{bestpred5}) again.
\end{theorem} 
 The proof of this assertion is rather simple using the same property for the standard conditional expectation. Let us now consider the prediction error in the $f$-distance. This corresponds to the notion of $f$-variance given by
\begin{definition}\label{fvar1}
Let $X\in\cL_f^2.$ The error of predicting $X$ by its $f$-mean is given by
\begin{equation}\label{fvar2}
\sigma^2_f(X) = d_f(X,E_f[X])^2 = E\big[\big(f(X) - f(E_f[X])\big)^2\big] =  E\big[\big(f(X) - E[f(X)]\big)^2\big].
\end{equation} 
\end{definition}
This is the standard variance of the variable transported onto $u(\cI).$ Similarly, we have a mixed total variance identity. For that, recall the standard notation for the conditional variance:
$$\sigma^2(X|\cG) = E[X^2|\cG]-\big(E[X|\cG])\big)^2 = E\big[\big(X - E[X|\cG]\big)^2|\cG\big]$$
Using this we also have:
\begin{theorem}\label{tvar1}
With the notations introduced above we have
\begin{equation}\label{tvar2}
\sigma_f^2(X) = E\big[\sigma_f^2(X|\cG)\big] + \sigma_f^2\big(E_f[X|\cG]\big).
\end{equation}
\end{theorem}
\begin{proof} 
It is just a computation. Begin with $\sigma^2_f(X) = E[f(X)^2]-\big(E[f(X)]\big)^2.$ Now note that
$$E[f(X)] = E[E[f(X)|\cG]] =  E[f\big(E_f[X\cG]\big)].$$
And now note that
$$E[f(X)^2] = E\big[E[f(X)^2|\cG]\big] = E\big[E[f(X)^2|\cG]\big] - E\big[f\big(E_f[X\cG]\big)^2] + E\big[f\big(E_f[X\cG]\big)^2].$$
Therefore 
$$\sigma_f^2(X) = E[\sigma^2_f(X|\cG)] + \sigma_f^2(E_f[X|\cG]).$$
\end{proof}
To close, we examine the notion of independence.  
It is natural to put
\begin{definition}\label{indep1}
Let $X\in\cL_f^1$ and let $\cG$ be a sub-$\sigma$-algebra of $\cF.$ $X$ is said to be $f$-independent of $\cG$ if and only if $E_f[X|\cG]=E_f[X].$
\end{definition}

Regretfully, except in the case when $cI=(0,\infty)$ and $f(x)=ln x,$ it is not true that this notion is equivalent to 
$$X\;\;\mbox{is said to be}\;f\;\mbox{independent of}\;\;\cG\;\;\Leftrightarrow \;\; E_f[XY]=E_f[X]E_f[Y]$$
for every bounded $\cG$-measurable $Y.$ But this may be a reason to use the logarithmic utility function when considering positive random variables.

 \subsection{Examples of arithmetic $f$-means}
In Table \ref{tab1} we compile a short list of simple examples. It is based on some of the standard utility functions
\begin{table}[h!]
\large
\centering
\begin{tabular}{|c c c|}\hline
$cI$         &          $f$  &  $E_f[X]$ \\\hline 
$[0,1]$  &    $x^{a}$   &   $\big(E[X^{a}]\big)^{1/a}$\\ 
$(0,\infty)$ &   $-1/x$ & $\big(E[1/X]\big)^{-1}$\\ 
$(0,\infty)$  &  $1-e^{-ax}$ & $-(1/a)\ln\big(E[e^{-aX}]\big)$\\ 
$\bbR$  &  $e^{ax}$ & $(1/a)\ln\big(E[e^{aX}]\big)$\\ 
$\bbR_+$ & $\ln x$  & $\exp\Big(E[\ln X]\Big)$\\
$\bbR$   & sinh(x)      &   $\sinh^{-1}\big(E[\sinh(X)]\big)$\\
$\bbR$   & $\Phi(x)$    & $q\big(E[\Phi(X)]\big)$\\\hline
\end{tabular}
\caption{Simple examples}
\label{tab1}
\end{table}

 Here $\Phi(x)$ is the cumulative distribution of a $N(0,1)$( or if you prefer, that of any other random variable with strictly increasing continuous distribution function), and $q(u):(0,1)\to\bbR$ is its associated quantile function (that is, the compositional inverse of $\Phi$).

\section{Application: Preferences under prior information}
Consider the situation of a decision maker, with utility function $u,$ that has to decide between cash flows $X$ and $Y$ under the presence of information given by a sub$-\sigma-$algebra $\cG.$ 
\begin{definition}\label{cchoice}
With the notations just introduced, we say that the decision maker prefers $Y$ to $X$ given $\cG$ whenever
\begin{equation}\label{cchoice}
 E_u[X|\cG] \leq E_u[Y|\cG].
\end{equation}
\end{definition}
Clearly, (\ref{cchoice}) is equivalent to $E[u(X)|\cG]\leq E[u(Y)|\cG].$ What is interesting here, is the following consequence of the tower property (ref{tower}) of conditional expectations. Consider a decision maker, that has information described by $\cG_1\subset\cG_2.$ Based on the information provided by $\cG_2,$ the decision maker prefers $Y$ to $X,$ that is $E_u[X|\cG_2] \leq E_u[Y|\cG_2].$ The decision maker wants to know whether basing his decision on the information provided by $\cG_1$ is consistent with this. The tower property of conditional expectation asserts that it is, because:
$$E_u[X|\cG_1] = E_u[E_u[X|\cG_2]|\cG_1] \leq E_u[E_u[Y|\cG_2]|\cG_1] = E_u[Y|\cG_1].$$

\section{Another potentially useful application: the conditional certainty equivalent}
This very short section contains an application of the notion of $f$-conditional expectation to the updating of the certain value when partial information becomes available. Instead of using the letter $f,$ let us switch to the $u$ of utility. 

So, let us suppose that we have a two time cash flow, given by $X_k$ at time $t_k,$ with $0 < t_1<...<t_m=T.$  And let $\cG_1\subset\cG_2\subset...\cG_m\subseteq\cF$  be a sequence of sub-$\sigma$-algebras such that $X_k$ is measurable with respect to $\cG_k.$ To simplify we might consider $\cG_k=\sigma(X_1,...,X_k$ the sub$-\sigma$-algebra generated by $\{X_1,...,X_k\}.$ The investor or decision maker wants to make decisions on his project at the chosen times $t_1<...,t_{m-1}\}$ within his investment time horizon $T,$ at which he monitors his project.  

The  investor has utility function $u$ such that $u(X_k)$ is integrable, and regards the certain value of the terminal cash flow as the price of the project at time $t=0:$
$$C(T) = u^{-1}\Big(E[u(X_T)]\Big).$$
At time $t_k$ the investor has recorded values $X_1,...,X_k$ of the project. We propose that a way to define the certainty equivalent for the remainder of the life of the project by the conditional certainty equivalent given by
\begin{equation}\label{conseq}
C(T|\cG_k) = E_u[X_T|\cG_k]= u^{-1}\Big(E[u(X_T)|\cG_k]\Big).
\end{equation}
Note in passing that when $u$ is the utility function of a risk averse investor (that is, when it is concave), from Jensen's inequality it follows that
$$C(T|\cG_k) = u^{-1}\Big(E[u(X_T)|\cG_k]\Big) \leq E[X_T|\cG_k]$$
and the extended Pratt's risk premium (Pratt \cite{Pr}) can be defined by
$$E[X_T|\cG_k]-C(T|\cG_k) = E[X_T|\cG_k] - E_u[X_T|\cG_k].$$
If we want to use the certain prices given by (\ref{conseq}) as prices for the project, the use is consistent with the Samuelson's dictum -Samuelson \cite{Sam} -that fair prices have to be modeled by martingales.\footnote{See the Theorem of Fair Game Futures Pricing} But in this case, they are $u$-martingales.
\begin{theorem}\label{mart1}
Let $u$ be a utility function as just mentioned in the previous paragraphs. Then the sequence 
\{$\pi_k=C(T|\cG_k):k=0,1,...,m\}$ is an a $u$-martingale with respect to $\{\Omega,\{\cG_k\}_{\{k=0,1,...,m\}},\bbP\},$ with value $\pi_0=C(T)$ at $t=0.$ 
\end{theorem}
Again, the proof is rather easy and based on the tower property (\ref{tower}) for $u$-martingales. A possible extension of this simple model goes as follows. Let $\{(W_n,X_n):n\geq 0\}$ be a positive, bivariate process, describing the wealth of an investor and the cash flow of some project. As above, suppose that the process is defined on $(\Omega,\cF,\{\cG_n\}_{n\geq 0}, P)$ and that $(W_n,X_n)$ is measurable with respect to $\cG_n.$ If $u$ denotes the utility function of the investor, we can define the certainty equivalent of the risk $X$ by
\begin{equation}\label{conseq2}
u(W_n+C(T|\cG_n)) = E[u\big(W_T+X_T\big)|\cG_n] \Leftrightarrow C(T|\cG_n) = E_u(W_T+X_T)|\cG_n] - W_n.
\end{equation}

\subsection{Example of sequential certainty equivalents}
A nice situation exists when the payoffs on Section 4 are supposed to be modeled by a discrete time Markovian process. In this case one can further compute the conditional certainty values. If 
$$Pu(x)=\int_{\cI}u(\xi)P(d\xi,x)$$
describes the action of the one-step transition probability upon $u,$ then if the observation times are $t_k=k$ for $k=1,...,N,$ we can explicitly rewrite (\ref{conseq}) as
$$C(N|\cG_k) = P_{(N-k)}u(X_k), \;\;\;k=0,1,...,N-1,$$
\noindent where $P_j$ denotes the iteration of $P$ $j$-times. In this might make use of the Markov structure of the cash flow process to devise a decision process to opt out of a project when its certain value $C(N|\cG_k)$ fall some preassigned level $L.$ For this define
$$T_L = min\{k: C(N|\cG_k)<L\}, \;\;\mbox{or}\;\;\infty \;\;\mbox{whenever}\;\;C(N|\cG_k)\geq L$$
and estimate the probability that $T_L$ is lower than a certain investment horizon.

\section{Estimating the generalized, arithmetic $f$-mean}
As above, $f:\cI\to\bbR$ is a given, continuous increasing function, and $d_f$ denotes the distance it induces on $\cI$ or on the class of $\cI$-valued $f$-integrable random variables. Throughout this section $\{X_n: n\geq 1\}$ denotes a collection of $i.i.d$-$\cI$ valued random variables, and denote by $X$ a random variable with their common distribution. Suppose that $X\in\cL_f^2.$  From Section 2, in particular, from (\ref{vmean}), the following definition of estimator of the $f$-mean value is clear.
\begin{theorem}\label{empmean1}
With the notations introduced above, the estimator of the $f$-mean of $X$ defined in
\begin{equation}\label{empmean2}
\hat{m}_N(X) = f^{-1}\Big(\frac{1}{N}\big(f(X_1) + ... + f(X_N)\Big)
\end{equation}
is an unbiased estimator of $E_f[X].$
\end{theorem}
The proof is a simple computation. Just invoke (\ref{bestpred5}):
$$E_f[\hat{m}_N(X)]=f^{-1}\Big(E\big[\frac{1}{N}\sum_{k=1}^Nf(X_k)\big]\Big) = f^{-1}\Big(E\big[f(X)\big]\Big) = E_f[X].$$
 We also have
\begin{theorem}\label{lln1}
Under the assumptions made at the beginning, $\hat{m}_N(X) \to E_f[X]$ except for a $\bbP$-null set.
\end{theorem}
\begin{proof}
Under the assumptions on the sequence $\{X_n:n\geq1\},$ the strong law of large numbers holds for the sequence $\{f(X_k):k\geq1\}$, that is,
$$\frac{1}{N}\big(f(X_1) + ... + f(X_N) \rightarrow E[f(X)],\;\;\;\bbP-\mbox{a.s.}$$
Invoke the continuity of $f$ to finish the proof. For full detail about the LLN see either Jacod and Protter \cite{JP} of Borkhar \cite{Bork}. 
\end{proof}  
{\bf Comment:} A similar result holds when conditioning with respect to a sub-$\sigma$-algebra $\cG.$ To complete, we state the analogue of the Central Limit Theorem in the $f-$distance. The proof can be seen in \cite{JP} for example.
\begin{theorem}\label{clt}
Let $\{X_n|n\geq 1\}$ be and i.i.d. $\cI$ valued random variables with finite $f-$variance (denoted by $\sigma_f$). Put $\mu_f=E_f[X]$ and $S_n=\sum_{k=1}^n(f(X_k)-\mu_f).$ Then 
$$Z_n = \frac{1}{\sigma_f\sqrt{n}}S_n\,\,\,\mbox{converges in distribution to}\;\; Z$$
where $Z$ that has an $N(0,1)$ distribution.
\end{theorem}

\section{Some extensions} 
Let us first examine rapidly some obvious extensions to the material in Section 2. With the notation introduced there, for a collection $\{x_1,...,x_n\}$ of points in $cI$ and a collection $\{w_1,...,w_n\}$ of positive weights, define the weighted (squared) $f-$distance of the collection to a point $m\in \cI$ by
$$\sum_{k=1}^nw_k(f(x_k)-f(m))^2.$$
A repetition of the same argument shows that the $m^*$ that minimizes the previous expression is
\begin{equation}\label{empmean4}
m^* = f^{-1}\Big(\sum_{k=1}^n\frac{w_k f(x_k)}{\sum_{k=1}^n w_k}\Big)
\end{equation}
If instead of a finite collection of points, we have a $\cI-$ valued random variable $X$ with distribution function $P(dx).$ Instead of a collection of weights we have some positive function $w(x)$ on $\cI,$ we might similarly define the weighted distance between $X$ and a constant $m\in\cI$ by
$$\int_\cI w(x)\big(f(x)-f(m)\big)^2P(dx)$$
Similarly, the point $m$ that minimizes this expression is
\begin{equation}\label{empmean5}
x^* = f^{-1}\Big(\frac{1}{\int w(x)P(dx)}\int w(x)f(x)P(dx)\Big)
\end{equation}
Below, $\cI$ and $\cJ$ denote intervals such that $f:\cI\to\cJ$ and $f^{-1}:\cJ\to\cI$ are strictly increasing bijections. The next result relates the best predictors in the distances determined by $f$ and 
$f^{-1}$ when $f$ is a concave (resp. convex) function.
\begin{theorem}\label{ext1}
With the notations just introduced, let $(\Omega,\cF,P)$ be a probability space and $\cG\subset\cF$ be a sub$-\sigma-$algebra. Let $X$  be a random variable taking values either in $\cI$ or $\cJ.$ Then:\\
{\bf a} If $f$ is concave we have:
$$E_f[X|\cG] \leq E[X|\cG] \leq E_{f^{-1}}[X|\cG].$$
{\bf b} If $f$ is convex we have
$$E_{f^{-1}}[X|\cG] \leq E[X|\cG] \leq E_{f}[X|\cG].$$
\end{theorem}
The proof of the statements is a direct application of Jensen's inequality. Clearly, when $f$ is neither concave nor convex on its domain, such results are not available. Consider for example $f(x)=x^3$ defined on the whole real line. The predictors are defined, but there may not exist a comparison among them. 

At this point, we refer the reader once more to examples 4 (with $a=1$) and 5 in Table \ref{tab1}. This example was used recently by Bauer and Zanjani \cite{BZ} in the context of risk capital allocation.

A further extension consists in examining compositions like $f\circ g$ or $g\circ f$ in which $g$ has the appropriate convexity  (concavity) properties such that the inequalities in Theorem \ref{ext1} are preserved. What we need at this point is motivational examples.

\section{Closing comments}
To sum up, not only can generalized arithmetic means be interpreted as minimizers of a distance to a given set of points, that is, in the geometric setting they can be interpreted as a best predictors. Since certainty values are generalized means, in the geometric setting they become best predictors of cash flows. And this combination of roles: predictors and certainty equivalents can be used to define conditional preference criterion as well as a sequence of ``fair'' prices of a random cash flow.

\end{document}